\documentclass{amsart}
\usepackage{ae,aecompl}\usepackage{ifthen}
\usepackage{pifont}
\usepackage[utf8]{inputenc}
\usepackage{amsrefs}
\usepackage{tabularx}
\usepackage{graphicx}
\usepackage[a4paper,twoside]{geometry}
\DeclareMathAlphabet{\mathpzc}{OT1}{pzc}{m}{it}
\newboolean{dum}
\setboolean{dum}{false}
\DeclareMathOperator{\Tr}{Tr} 
\newtheorem{lemma}[subsection]{Lemma}
\newtheorem{prop}[subsection]{Proposition}
\newtheorem{thm}[subsection]{Theorem}
\newtheorem{cor}[subsection]{Corollary}
\newtheorem*{conj}{Conjecture}
\newcommand{\comment}[1]{\ifthenelse{\boolean{dum}}{
{\par\noindent\Huge\ding{46}} \fbox{\parbox{10cm}{#1}}\par}{}}
\DeclareMathOperator{\card}{card}
\DeclareMathOperator{\vol}{vol}
\newcommand{\qdisc}{\mathpzc{q}}
\newcommand{\pdisc}{\mathpzc{p}}

\newcommand{\tdisc}{t}
\newcommand{\xdisc}{x}
\newcommand{\ydisc}{y}
\newcommand{\zdisc}{z}
\newcommand{\phidisc}{\phi} 
\newcommand{\Wdisc}{\mathpzc{W}}
\newcommand{\Kdisc}{\mathpzc{K}}
\newcommand{\Xdisc}{\mathpzc{X}}
\newcommand{\Qdisc}{\mathpzc{Q}}
\newcommand{\Qcont}{Q}
\newcommand{\qcont}{q}

\newcommand{\tcont}{t}
\newcommand{\xcont}{x} 
\newcommand{\ycont}{y}
\newcommand{\phicont}{\varphi}
\newcommand{\Wcont}{W}
\newcommand{\Xcont}{\mathbf{X}}
\newcommand{\Kcont}{\mathbf{K}}
\newcommand{\px}{x}

\begin{document}
\bibliographystyle{alpha}
\comment{
$ $Id: article.tex,v 1.21 2008/02/19 05:42:32 enord Exp $ $
}
\title{On the Shuffling Algorithm for Domino Tilings}
\author{Eric Nordenstam}
\begin{abstract}
We study the dynamics of a certain discrete
model of interacting  particles that comes from
the so called shuffling algorithm for sampling a
random tiling of an Aztec diamond. 
It turns out that the transition probabilities
have a particularly convenient determinantal form.
An analogous formula in a continuous setting 
has recently been obtained by Jon Warren
studying certain model of  interlacing Brownian motions
which can be used to construct Dyson's non-intersecting
Brownian motion.

We conjecture that Warren's model can be recovered as
a scaling limit of our discrete model and 
prove some partial results in this direction. 
As an application to one of these results  we use it 
to  rederive the known 
result that random tilings of an Aztec diamond, 
suitably rescaled near a turning point, converge
to the GUE minor process. 
\end{abstract}

\maketitle
\section{Introduction}
\label{sec:introduction}
There has been a lot of work in recent years connecting
tilings of various planar regions with random matrices. 
One particular model that has been intensely studied 
is domino tilings of a so called \emph{Aztec diamond}. 
One way to analysing that model,
\cites{johansson:arctic_circle,johansson:discrete_orthogonal,
johansson:gue_minors}, 
 is to define a particle process
corresponding to the tilings so that uniform
measure on all tilings induces some measure on this particle process.

In this article we will study the so called \emph{shuffling algorithm},
described in 
\cites{elkies:alternating_sign_matrices_II,propp:generalized_domino},
which in various variants
 can be used either to count or to enumerate all tilings of 
the Aztec diamond or to sample a random such tiling.

The sampling of a random tiling by this method is an iterative process.
Starting with a tiling of an order $n-1$ 
Aztec diamond, a certain procedure is performed, 
producing a random tiling of order $n$. 
This procedure is usually described in terms of the dominoes
which should be moved and created according to a certain procedure.
We will instead look at this algorithm as a certain 
dynamics on the particle process mentioned above. 

The detailed dynamics of the particle process will be presented in
section~\ref{sec:particle-process} and how it is obtained from the traditional
formulation of the shuffling algorithm is 
presented in section~\ref{sec:shuffling-algorithm}.
For now, consider a process  $\Xdisc(t)=(X^1(t), \dots, X^m(t))$ 
for $t=0$, $1$, $2$, \dots, where 
$X^k(t)=(X^k_1(t),\dots, X^k_k(t))\in \mathbb{Z}^{k}$. 
The quantity $X^j_i(t)$ represents the position of the $i$:th particle
on line $j$ after  $t-j$ steps of the shuffling algorithm have
been performed. (The reason for the $t-j$ is technical
convenience.)
We will  show that 
\begin{thm}
\label{thm:dyson_bm}
For fixed  $k$, consider only the component
 $X^k(t)$ from $\Xdisc(t)$
rescaled according to 
\begin{equation}
\tilde X^n_i(t) = \frac{X^n_i(Nt) - \frac{1}{2} Nt}{\frac{1}{2}\sqrt{N}}
\end{equation}
and defined by linear interpolation for non-integer values of $Nt$. 
The process $ \tilde X^n(t)$ converges to a Dyson Brownian motion
with all particles started at the origin as 
$N\rightarrow\infty$, 
in the sense of convergence of
 finite dimensional distributions. 
\end{thm}

The full process $(\Xdisc(t))_{t=0,1,\dots}$ has remarkable
similarities to, and is we believe a discretization of,
a process studied recently by Warren, \cite{warren:dyson_brownian_motions}.
It consists of many interlaced Dyson Brownian motions and is
here briefly described in section~\ref{sec:warren}. 
We will denote that process $(\Xcont(t))_{t\geq 0}$.
There is reason to believe the following.
\begin{conj}
\label{thm:asymptotics}
Consider  the process $(\Xdisc(t))_{t=0,1,\dots}$
rescaled according to 
\begin{equation}
\tilde X^n_i(t) = \frac{X^n_i(Nt) - \frac{1}{2} Nt}{\frac{1}{2}\sqrt{N}}
\end{equation}
and defined by linear interpolation for non-integer values of $Nt$. 
The process $\tilde \Xdisc(t)$ converges to Warren's process $\Xcont(t)$ 
as
$N\rightarrow\infty$, 
in the sense of convergence of
 finite dimensional distributions. 
\end{conj}

The key to our asymptotic analysis of the shuffling algorithm
is that the transition probabilities of $(X^k, X^{k+1})$ can be written
down in a convenient determinantal form, see proposition \ref{thm:xfer}. 
These formulas mirror beautifully  formulas obtained by
 Warren.

As an application of our results we will 
use it to rederive an asymptotic result about 
random tilings near the point where the arctic circle 
touches the edge of the diamond.  
This result was first stated in~\cite{johansson:arctic_circle}
and proved in~\cite{johansson:gue_minors}. 

Recall that the \emph{Gaussian Unitary Ensemble}, or GUE for short,
is a probability measure on Hermitian matrices
with density $Z_n^{-1} e^{-\Tr H^2/2}$ 
where $Z_n$ is a normalisation constant that depends on 
the dimension $n$ of the matrix.
Let $H=(h_{rs})_{1\leq r,s\leq n}$  a GUE matrix and
denote its principal minors by
$H_j=(h_{rs})_{1\leq r,s\leq j}$.
Let $\lambda^j=(\lambda^j_1,\dots, \lambda^j_j)$ be the 
eigenvalues of $H_j$. 
Then  $\Lambda=(\lambda^1,\dots,\lambda^n)\in \mathbb{R}^{n(n+1)/2}$
is the so called GUE minor process.

\begin{thm}[Theorem 1.5 in \cite{johansson:gue_minors}.]
\label{thm:GUE_process}
Let the $\mathbb{R}^{n(n+1)/2}$-valued  process
 $\tilde\Xdisc(t)=(\tilde X^1(t), \dots, \tilde X^n(t))$
 be a rescaled version
of $\Xdisc(t)$ with 
\begin{equation}
 \tilde X^j_i(t)= \frac{X^j_i(t) - \frac{t}{2}}{\frac{1}{2}\sqrt{t}}.
\end{equation}
Then $\tilde\Xdisc(t)\rightarrow \Lambda$ as $t\rightarrow \infty$
in the sense of weak convergence of probability measures. 
\end{thm}

To put this in perspective, let us note that a similar result 
for lozenge tilings 
is known from Okounkov and Reshetikhin~\cite{okounkov:birth}.
They discuss the fact that, for quite general regions,
that close to a so called turning point 
the GUE minor process can be obtained in a limit. 
A turning point is, just as in our situation, 
 where the disordered region is tangent to the domain boundary.

\section{The Aztec Diamond Particle Process}
\label{sec:particle-process}
We will here content ourselves with stating the rules
of the particle dynamics that we will study.
The reader will in section~\ref{sec:shuffling-algorithm} find
a description the traditional formulation of the shuffling algorithm
and how that relates to the formulas below.

Consider the process  $(\Xdisc(t))=(X^1(t), \dots, X^n(t))$ 
for $t=0$, $1$, $2$, \dots, where 
$X^k(t)=(X^k_1(t),\dots, X^k_k(t))\in \mathbb{Z}^{k}$.
It satisfies the initial condition
\begin{equation}
X^k(0)= \bar x^k
\end{equation}
 where $\bar x^j_i=i$ for $1\leq i\leq j$.
At each time $t$ 
the process fulfils the interlacing condition
\begin{equation}
\label{eq:4}
X^{k}_i(t) \leq X^{k-1}_i(t) < X^{k}_{i+1}(t)
\end{equation}
and evolves in time according to 
\begin{equation}
\label{eq:1}
\begin{aligned}
X^1_1(t)&=X^1_1(t-1)+\beta^1_1(t)  \\
X^j_1(t)&=X^j_1(t-1)+\beta^j_1(t) \\
&\quad - \mathbf{1}\{X^j_1(t-1)+\beta^j_1(t)= X^{j-1}_1(t)+1\}
&&\text{for $j\geq 2$}\\
X^j_j(t)&=X^j_j(t-1)+\beta^j_j(t)\\
&\quad +\mathbf{1}\{X^j_1(t-1)+\beta^j_1(t)= X^{j-1}_{j-1}(t)\}
&&\text{for $j\geq 2$}\\
X^j_i(t)&=X^j_i(t-1)+\beta^j_i(t)\\
&\quad- \mathbf{1}\{X^j_i(t-1)+\beta^j_i(t)= X^{j-1}_i(t)+1\}\\
&\quad+\mathbf{1}\{X^j_i(t-1)+\beta^j_i(t)= X^{j-1}_{i-1}(t)\}
&&\text{for $j\geq 3$ and $1<i<j$.}
\end{aligned}
\end{equation}
for $t=1$, $2$, \dots where all the $(\beta^j_i(t))_{i,j, t}$ are i.i.d.
unbiased coin tosses, satisfying
$\mathbb{P}[\beta^1_1(1)=0]=\mathbb{P}[\beta^1_1(1)=1]=\frac{1}{2}$.

One way to think about this is that at each time $t$,
this is 
a set of particles on $n$ lines.
The $k$:th line has $k$ particles
on it at positions $X^k_1$, \dots, $X^k_k$. 
At each time step each of these particles either stays or
jumps one unit step forward independent of all others except
that the particles on line $k$ can
force particles on line $k+1$ to jump or to stay to 
enforce the the interlacing condition~(\ref{eq:4}).
Also note that the interlacing implies that $X^k_i<X^k_{i+1}$
at each time $t$, i.e. two particles cannot occupy the same
space at the same time.

As mentioned  we can write down transition probabilities
for this process on a particularly convenient determinantal form.
Define 
$\delta_i:\mathbb{Z}\rightarrow\mathbb{Z}$
such that $\delta_i(x)=1$ if $i=x$ and  $\delta_i(x)=0$ otherwise.
Let us first introduce some notation.

\begin{align*}
(\phidisc*\psi)(\xdisc)&=\sum_{s+t=\xdisc}\phidisc(s)\psi(t)&&\text{(Convolution product)}\\
\phidisc^{(0)}&=\delta_0\\
\phidisc^{(n)}&=\phidisc^{(n-1)}*\phidisc &&\text{for $n=1,2,\dots$}\\
\Delta\phidisc&=(\delta_0 - \delta_{1})*\phidisc&&\text{(Backward difference)}\\
\Delta^{-1}\phidisc(\xdisc)&=\sum_{\ydisc=-\infty}^\xdisc \phidisc(\ydisc)\\
\bar\Delta\phidisc&=(-\delta_0 + \delta_{-1})*\phidisc&&\text{(Forward difference)}\\
\bar\Delta^{-1}\phidisc(\xdisc)&=\sum_{\ydisc=-\infty}^{\xdisc-1} \phidisc(\ydisc)\\
\end{align*}

Let 
 $\Wdisc^{n+1,n}= \{ (x,y):x_1\leq y_1<x_2\leq \dots\leq y_n<x_{n+1} \}
\subset\mathbb{Z}^{n+1}\times\mathbb{Z}^n$. 
For $(\xdisc,\ydisc),(\xdisc',\ydisc')\in \Wdisc^{n+1,n}$
and $t=0,1,\dots$, define
\begin{equation}
\qdisc_t^n((\xdisc,\ydisc), (\xdisc',\ydisc'))=
  \det
\begin{bmatrix}
A_\tdisc(\xdisc,\xdisc')& B_\tdisc(\xdisc,\ydisc')\\
C_\tdisc(\ydisc,\xdisc')& D_\tdisc(\ydisc,\ydisc')
\end{bmatrix}
\end{equation}
where 
\begin{itemize}
\item $A_t(\xdisc,\xdisc') $ is an $(n+1)\times(n+1)$-matrix 
where element $(i,j)$ is $\phidisc^{(t)}(\xdisc_i'-\xdisc_j)$,
\item $B_t(\xdisc,\ydisc') $ is an $(n+1)\times(n)$-matrix 
where element $(i,j)$ is 
$\Delta^{-1}\phidisc^{(t)}(\ydisc_i'-\xdisc_j)-\mathbf{1}\{j\geq i\}$,
\item $C_t(\ydisc,\xdisc') $ is an $n\times(n+1)$-matrix 
where element $(i,j)$ is $\Delta \phidisc^{(t)}(\ydisc_i'-\xdisc_j)$ and
\item $D_t(\ydisc,\ydisc') $ is an $n\times n$-matrix 
where element $(i,j)$ is $\phidisc^{(t)}(\ydisc_i'-\ydisc_j)$.
\end{itemize}

Let $\Wdisc^{n}= \{x:x_1<x_2<\dots<x_n\}\subset\mathbb{Z}^n$ and
for $x\in \Wdisc^{n}$ let 
\begin{equation}
h_n(x)=\prod_{i<j}(x_j-x_i).
\end{equation}
Finally, after all this notation, we can state a result.

\begin{thm}
\label{thm:transition}
The transition probabilities of 
$(X^k, X^{k+1})$ from the process $\Xdisc$ above 
are 
\begin{equation}
\label{eq:11}
\qdisc_t^{k,+}((\xdisc,\ydisc), (\xdisc',\ydisc')):=
\frac{h_k(y')}{h_k(y)} \qdisc_t^{k}((\xdisc,\ydisc), (\xdisc',\ydisc'))
\end{equation}
that is
\begin{equation}
\mathbb{P}[(X^{k+1}(s+t), X^{k}(s+t))=(x',y') ; (X^{k+1}(s), X^{k}(s))=(x,y)  ]
= \qdisc_t^{k,+}((\xdisc,\ydisc), (\xdisc',\ydisc')).
\end{equation}
\end{thm}
A proof is given in section~\ref{sec:interl-from-aztec} 
and the reason I defined
$\qdisc_t^{k}$ as opposed to defining $\qdisc_t^{k,+}$ directly
will become obvious in the next section.

Given the exact expressions above it is 
a very straightforward computation to 
 integrate out the $x$ component in 
expression~(\ref{eq:11}). We  find that 
the transition probabilities of 
$(X^k)$ from the process $\Xdisc$ above 
is  
\begin{equation}
\label{eq:15}
\pdisc_t^{k,+}(\ydisc, \ydisc'):=
\frac{h_k(y')}{h_k(y)} \pdisc_t^{k}(\ydisc, \ydisc')
\end{equation}
where $\pdisc_t^k(y,y'):=D_t(y, y')$ given above. 
We recognise this as the transition probability
for random walks conditioned never to intersect, a fact
that is so important we state it properly. 
\begin{cor}
\label{thm:discrete_dyson_bm}
The component $X^k(t)$ of $\Xdisc(t)$ is a discrete Dyson Brownian motion
of $k$ particles started at $\bar x^k$.
\end{cor}
This fits nicely with theorem~\ref{thm:dyson_bm}. 
The component $X^k$ from $\Xdisc$ simply $k$ simple symmetric 
random walks conditioned never to intersect, their limit is 
$k$ Brownian motions conditioned never
to intersect, which is exactly what $X^k$ from Warren's process $\Xcont$ is.

\section{Interlacing Brownian motions}
\label{sec:warren}
\renewcommand{\tdisc}{\mathpzc{t}}
We will now digress a bit and summarise Warren's work 
in~\cite{warren:dyson_brownian_motions},  so as
to see the similarities between his continuous process
and our discrete process.  
The reader
is referred to that reference for more details of the construction. 
Consider  an $\mathbb{R}^{n+1}\times\mathbb{R}^{n}$-valued
stochastic process $(\Qcont(t))_{t\geq 0}=(X(t),Y(t))_{t\geq 0}$
satisfying an interlacing condition 
\begin{equation}
X_1(t)\leq Y_1(t)\leq X_2(t)\leq \dots \leq Y_n(t)\leq X_{n+1}(t),
\end{equation} 
and equations
\begin{align}
Y_i(t)&=y_i + \beta_i(t\wedge\tau),\\
X_i(t)&=y_i + \gamma_i(t\wedge\tau)+L_i^-(t\wedge\tau)- L_i^+(t\wedge\tau)
\end{align}
where

$(\beta_i)_{i=1}^n$ and $ (\gamma_i)_{i=1}^{n+1}$
are independent Brownian motions, 

$\tau=\inf\{t\geq 0: Y_i(t)=Y_{i+1}\text{ for some $i$}\}$,

$L^-_1\equiv L^+_{n+1}\equiv0$ and
\begin{align}
L_i^+(t)&=\int_0^t \mathbf{1}(X_i(s)=Y_i(s)) \, dL_i^+(s) &
L_i^-(t)&=\int_0^t \mathbf{1}(X_{i}(s)=Y_{i-1}(s)) \, dL_i^-(s)
\end{align}
are twice the semimartingale local times at zero of $X_i-Y_i$
and $X_i-Y_{i-1}$ respectively.

This process can be constructed by first constructing the Brownian
motions $\beta_i$ and $\gamma_i$ and then using Skorokhod's construction
to push $X_i$ up from $Y_{i-1}$ and down from $Y_i$. The process
is killed when $\tau$ is reached, i.e. when two of the $Y_i$ meet.

Warren then goes on to show that the transition densities
of this process have a determinantal form similar to 
what we have seen in the previous section. 
Let $\phicont_t(\xcont)=(2\pi t)^{-1/2} e^{-x^2/2t}$
and $\Phi_t(\xcont) = \int_{-\infty}^x \phicont_t(\ycont)\,d\ycont$.
Let 
$\Wcont^{n,n+1}=\{(x,y)\in\mathbb{R}^n\times \mathbb{R}^{n+1}:
x_1<y_1<x_2<\dots<y_n<x_{n+1}\}$. 

Define $\qcont_t^n((\xcont,\ycont),(\xcont', \ycont'))$ 
for $(\xcont,\ycont)$, $(\xcont',\ycont')\in \Wcont^{n,n+1}$
and $t>0$ to be the determinant of the matrix
\begin{equation}
\begin{bmatrix}
A_\tcont(\xcont,\xcont')& B_\tdisc(\xcont,\ycont')\\
C_\tcont(\ycont,\xcont')& D_\tdisc(\ycont,\ycont')
\end{bmatrix}
\end{equation}
where 

 $A_t(\xcont,\xcont') $ is an $(n+1)\times(n+1)$-matrix 
where element $(i,j)$ is $\phicont_{ t }(\xcont_i'-\xcont_j)$,

 $B_t(\xcont,\ycont') $ is an $(n+1)\times(n)$-matrix 
where element $(i,j)$ is 
$\Phi_{ t }(\ycont_i'-\xcont_j)-\mathbf{1}(j\geq i)$,

$C_t(\ycont,\xcont') $ is an $n\times(n+1)$-matrix 
where element $(i,j)$ is $\phicont_{ t}'(\ycont_i'-\xcont_j)$ and

 $D_t(\ycont,\ycont') $ is an $n\times n$-matrix 
where element $(i,j)$ is $\phicont_t(\ycont_i'-\ycont_j)$.

\begin{prop}[Prop 2 in \cite{warren:dyson_brownian_motions}]
The process $(X,Y)$ killed at time $\tau$ has transition densities 
$\qcont^n_t$, that is 
\begin{equation}
\qcont_t^n((\xcont,\ycont), (\xcont, \ycont))
\, d\xcont' d\ycont'=
\mathbb{P}^{x,y}[X_t\in d\xcont',Y_t\in d\xcont'; t<\tau]
\end{equation}
\end{prop}

Warren goes on to condition the $Y_i$ not
to intersect via so called the Doob $h$\nobreakdash-transform. 
The transition densities for the transformed process 
are given in terms of the those for the killed process by
\begin{equation}
\qcont^{n,+}_t((\xcont,\ycont),(\xcont', \ycont'))=
\frac{h_n(y')}{h_n(y)}
\qcont^n_t ((\xcont,\ycont),(\xcont', \ycont')).
\end{equation}
%

He also shows that you can start all the $X_i$ and $Y_i$ 
of the transformed process at the origin by giving a so called
entrance law,
\begin{equation}
\label{eq:3}
\nu_t^n(x,y):=\frac{n!}{Z_{n+1}}t^{-(n+1)^2/2}
\exp\left\{-\sum_i x_i^2/(2t)\right\}
\left\{ \prod_{i<j}(x_j-x_i)\right\}
\left\{ \prod_{i<j}(y_j-y_i)\right\},
\end{equation}
that is,  showing 
(lemma~4 of~\cite{warren:dyson_brownian_motions}) 
that this expression satisfies 
\begin{equation}
\nu_{t+s}^n(x',y')= 
\int_{\Wcont^{n, n+1}}
\nu_s^n(x, y) 
\qcont_t^{n,+} ((x,y),(x',y')) \, dxdy.
\end{equation}

It is possible to integrate out  the $X$ components in 
that transition density and entrance law.
The result is transition density  
\begin{equation}
p^{n,+}_t(y,y'): = \frac{h(y')}{h(y)}\det D_t(y, y')
\end{equation}
and entrance law
\begin{equation}
\label{eq:14}
\mu^n_t(y):= \frac{1}{Z_n}t^{-n^2/2}
\exp\left\{-\sum_i y_i^2/(2t)\right\}
\left\{ \prod_{i<j}(y_j-y_i)\right\}^2.
\end{equation}

Now comes the interesting part.
Let $\mathbf{K}$ be the cone of points 
$x=(x^1, \dots, x^n)$ where 
$x^k=(x^k_1,\dots, x^k_k)\in\mathbb{R}^k$.
Warren defines a process $\Xcont(t)$ taking values in $\mathbf{K}$
such that 
\begin{equation}
X^k_i(t)=x^k_i + \gamma^k_i(t) + L_i^{k,-}(t) -L_i^{k,+}(t)
\end{equation}
where the $(\gamma_i^k)_{i,k}$ are independent Brownian motions 
and $L_i^{k,+ }$ and $L_i^{k,+ }$ 
are continuous, increasing processes growing 
only when $X^k_i(t)=X^{k-1}_i(t)$ and $X^k_i(t)=X^{k}_{i-1}(t)$ 
respectively and the special cases
$L^{k,+}_k$ and $L^{k,-}_1$ are
identically zero for all $k$.

Think of this as essentially $n(n+1)/2$ particles performing 
independent  Brownian motions except that the 
$k$ particles in $X^{k}$  can push the particles in
$X^{k+1}$ up or down to enforce the interlacing condition 
that the whole process should stay in $\mathbf{K}$.

This full process process can be constructed inductively as follows.
\begin{enumerate}
\item The process $(X^k)$ has transition densities $p^{k,+}_t$ and entrance
law $\mu^k_t$.
\item
The process $(X^k,X^{k+1})$ has transition densities $q^{k,+}_t$ and entrance
law $\nu^k_t$.
\item
For $k=2, \dots, n-1 $ the process $(X^{k+1})$ is conditionally 
independent of $(X^1,\dots, X^{k-1})$  given $(X^k)$. 
\item
This implies (by some explicit calculations)
 that $(X^{k+1})$ has transition densities $p^{k+1,+}_t$ and entrance
law $\mu^{k+1}_t$.
\end{enumerate}
This argument shows that the following.

\begin{prop}[Warren]
\label{thm:xfer}
 There exists such a process
$\Xcont(t)$ started at the origin and it satisfies that  
for $k=1$, \dots, $n-1$, 
the process $(X^k, X^{k+1})$ has
entrance law $\nu_{t}^{n}$ and transition probabilities $\qcont^{k,+}$.
\end{prop}
It is this process $\Xcont$ that is 
the continuous analog of our discrete process $\Xdisc$.

\section{Shuffling algorithm}
\label{sec:shuffling-algorithm}

We will now show how relate some well known 
facts about sampling random tilings of an Aztec diamond
before showing how to get the particle dynamics in section
\ref{sec:particle-process}. 

The Aztec diamond of order $n$, denoted $A_n$, 
 is an area  in the plane 
that is the union of those lattice squares
 $[a, a+1]\times [b, b+1]\subset\mathbb{R}$ that 
are entirely contained in $\{|x|+|y|\leq n+1\}$.  
$A_n$ can be tiled in $2^{n(n+1)/2}$ ways by dominoes 
of size $2\times 1$.
We will be interested picking a random tiling. 
By random tiling in this article we will always mean 
that all possible tilings given the same 
probability.  

A key ingredient of almost all results concerning 
tilings of this shape is the realization that 
one can distinguish four kinds of dominoes present 
in a typical tiling.
The obvious distinction to the casual observer
is the difference between  horizontal
and vertical dominoes. These can be subdivided further.
Colour the underlying lattice squares black and white 
according to a checkerboard fashion in such a way that 
the left square on the top line is black. 
Let a horizontal domino be of type N or north 
if its leftmost  square is black, and of type S or south
otherwise. Likewise let a vertical domino 
be of type W or west if its  topmost square is black and
type E or east otherwise. In figures~\ref{fig:shuffle} and~\ref{fig:particles} 
the S and E type dominoes have been shaded for convenience.

One way of sampling from this measure 
is the so called shuffling algorithm, 
first described in~\cite{elkies:alternating_sign_matrices_II}, 
and very nicely explained and generalised in~\cite{propp:generalized_domino}.
It is an iterative procedure that 
given a random tiling of $A_n$
and some number of coin-tosses, 
produces a random tiling of a diamond of $A_{n+1}$.
You start with the empty tiling on $A_0$ and
you repeat this process until you have a tiling of the desired size.
It is a theorem that this procedure gives 
all tilings equal probability, provided that the coin-tosses
we have made along the way are fair.  

The algorithm works in three stages. Start with a 
tiling $A_n$.
\begin{description}
\item[Destruction]
All $2\times2$ blocks consisting of an  S-domino 
directly above an N-domino are removed. Likewise 
all $2\times2$ blocks of 
consisting of an E-domino directly to left of a W-domino 
are removed. 

\item[Shuffling]
All N, S, E and W-dominoes respectively  move one unit length
up, down, right and left respectively. 

\item[Creation]
The result is a tiling of a subset of $A_{n+1}$.
The empty parts can be covered in a unique
way  by 2$\times$2 squares.
Toss a coin to fill these with two 
horizontal or two vertical dominoes with 
equal probability.
\end{description}

Figure~\ref{fig:shuffle} illustrates the process. In the leftmost column 
there are tilings of successively larger diamonds.
From column one to column two, the destruction step is carried out.
From there to the third column, shuffling is performed. These figures
contain several dots which will concern us later in this
exposition. 
The creation step of the algorithm applied to a diamond in the last column
gives (with positive probability)
 the diamond in the first column on the next row.
\begin{figure}[htbp]
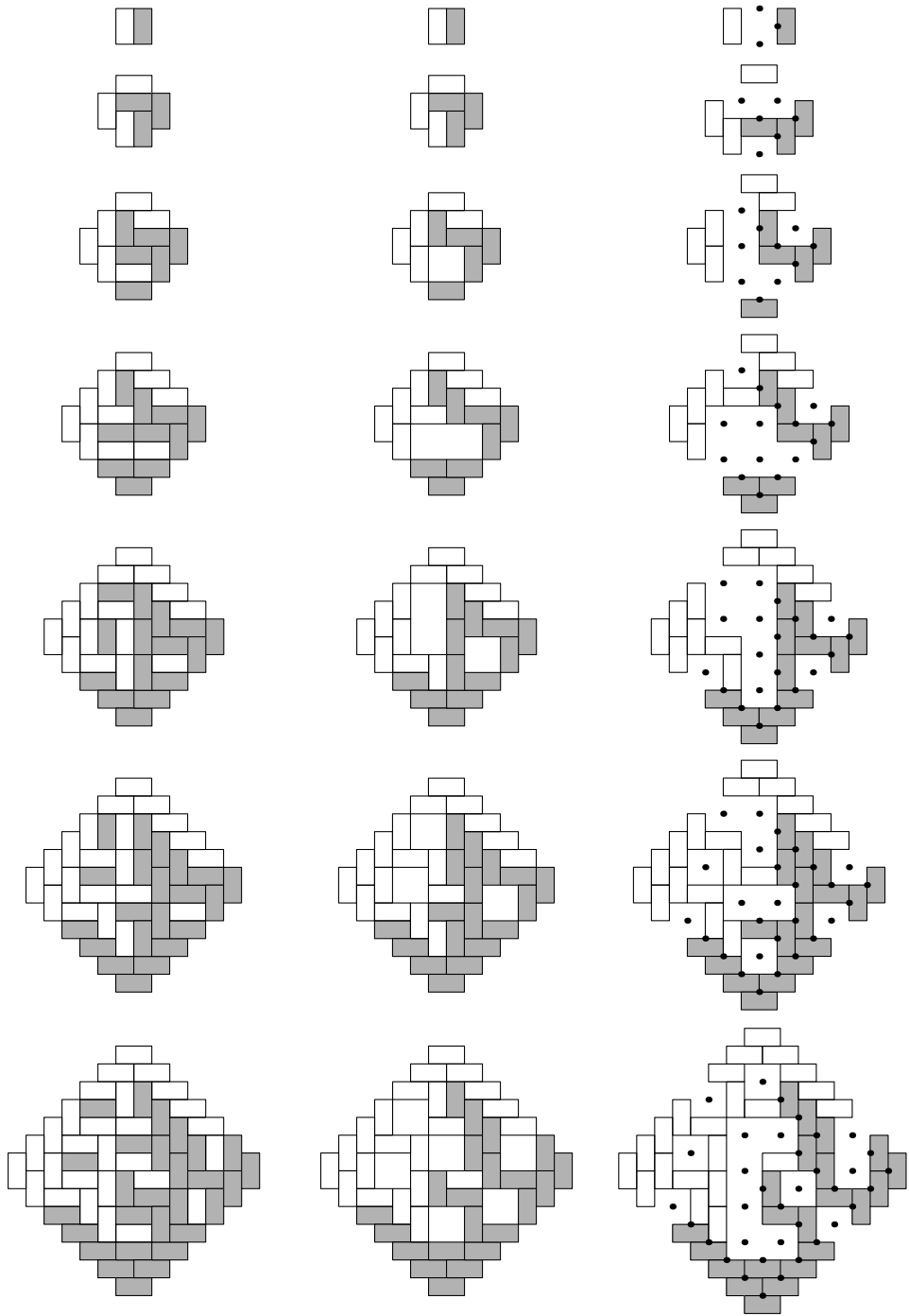

\noindent
\begin{tabular}{ m{4cm} m{4cm} m{4cm} r }
\center{\includegraphics{pictures/aztec.0.mps}}
&
\center{\includegraphics{pictures/aztec.1.mps}}
&
\center{\includegraphics{pictures/aztec.2.mps}} &\\
\center{\includegraphics{pictures/aztec.3.mps}}
&
\center{\includegraphics{pictures/aztec.4.mps}}
&
\center{\includegraphics{pictures/aztec.5.mps}} &\\
\center{\includegraphics{pictures/aztec.6.mps}}
&
\center{\includegraphics{pictures/aztec.7.mps}}
&
\center{\includegraphics{pictures/aztec.8.mps}} &\\
\center{\includegraphics{pictures/aztec.9.mps}}
&
\center{\includegraphics{pictures/aztec.10.mps}}
&
\center{\includegraphics{pictures/aztec.11.mps}} &\\
\center{\includegraphics{pictures/aztec.12.mps}}
&
\center{\includegraphics{pictures/aztec.13.mps}}
&
\center{\includegraphics{pictures/aztec.14.mps}} &\\
\center{\includegraphics{pictures/aztec.15.mps}}
&
\center{\includegraphics{pictures/aztec.16.mps}}
&
\center{\includegraphics{pictures/aztec.17.mps}} &\\
\center{\includegraphics{pictures/aztec.18.mps}}
&
\center{\includegraphics{pictures/aztec.19.mps}}
&
\center{\includegraphics{pictures/aztec.20.mps}} 
\end{tabular}
\caption{The shuffling procedure. S- and E-type dominoes are shaded.}
\label{fig:shuffle}
\end{figure}

To study more detailed properties 
of random tilings it is useful to introduce a coordinate system
suited to the setting and 
a particle process such that the possible tilings 
correspond to  particle configurations.

In the left picture in 
figure~\ref{fig:particles}, the S and E type dominoes are shaded and
a coordinate system is imposed on the tiling. 
For each tile there is exactly one of the $x$ lines and exactly one of
the $y$ lines that passes through its interior. 
Indeed we can uniquely specify the location
of a tile by giving its coordinates $(x,y)$ and type (N, S, E or W). 
You can see that along the line $y=k$ there are exactly $k$
shaded tiles, for $y=1\ldots 8$ where $8$ is the order of the diamond. 
The obvious generalisation of that statement 
is true for tilings of $A_n$ for any $n$. 
We shall call the occurrence of a shaded tile a particle.
The right picture in figure~\ref{fig:particles}
is the same tiling but with dots marking the particles.

Just to fix some notation, let $\px^j_i$ be the $x$-coordinate 
of the $i$:th particle along the line $y=j$.
It is clear from the definitions that these satisfy an interlacing criterion,
\begin{equation}
\px^{j}_i \leq \px^{j-1}_i \leq \px^{j}_{i+1}. 
\end{equation}
We will now see how the shuffling algorithm described above
acts on these particles. 

It turns out 
that the positions of the particles is uniquely determined before
the creation stage of the last iteration of the shuffling algorithm,
and we have marked these with dots in the last column
in figure~\ref{fig:shuffle}.
As can be seen in that figure, 
running the shuffling algorithm to produce tilings of  successively larger
Aztec diamonds imposes certain dynamics on these particles. 
That is the central object of study in this article.

Let us first consider the trajectory of $\px^1_1$.
As can easily be seen in figure~\ref{fig:shuffle},
on the $y=1$ line there are  always a number of W-dominoes,
then the particle, then a number of N-dominoes. 
Depending on whether the creation stage of the algorithm 
fills the empty space in between these with a pair of
horizontal or vertical dominoes, either the 
particle stays or its $x$-coordinate will increase
by one in the next step.
Thus the first particle
performs the simple random walk
\begin{equation}
\px^1_1(t)=\px^1_1(t-1)+\gamma_1^1(t).
\end{equation}
were $\gamma_j^i(t)$ are independent 
coin tosses, i.e. $P[\gamma^j_i(t)=1]=P[\gamma^j_i(t)=0]=\frac{1}{2}$, 
for $t, j=1,\ldots$ and $0\leq i\leq j$. 

Consider now the particles on row $y=2$. For $x^2_1$, 
while $\px^2_1(t)<\px^1_1(t)$ it performs a random walk independently
of  $\px^1_1$, at each time either staying or adding one with 
equal probability. 
However, when there is equality, $\px^2_1(t)=\px^1_1(t)$, 
then the particle must be represented by a vertical (S) tile.
Thus it does not contribute to growth of the west polar region,
thus the particle will remain fixed. In order to represent this as
a formula, we subtract one if the particle attempts to jump
past $x^1_1$. 
\begin{equation}
\px^2_1(t)=\px^2_1(t-1)+\gamma^2_1(t) 
- \mathbf{1}\{\px^2_1(t-1)+\gamma^2_1(t)= \px^1_1(t-1)+1\}
\end{equation}
Symmetry completes our analysis of this row with the
relation
\begin{equation}
\px^2_2(t)=\px^2_2(t-1)+\gamma^2_2(t)  + \mathbf{1}\{\px^2_2(t-1)+\gamma^2_2(t)= \px^1_1(t-1)\}.
\end{equation}

For the third row, our previous analysis applies to 
the first and last particle.
\begin{align}
\px^3_1(t)&=\px^3_1(t-1)+\gamma^3_1(t) - \mathbf{1}\{\px^3_1(t-1)+\gamma^3_1(t)= \px^2_1(t-1)+1\} \\
\px^3_3(t)&=\px^3_3(t-1)+\gamma^3_3(t) + \mathbf{1}\{\px^3_3(t-1)+\gamma^3_3(t)= \px^2_2(t-1)\}
\end{align}

On $y=3$ between $\px^2_1$ and $\px^2_2$ there must
be first a sequence of zero or more  E dominoes, then $\px^3_2$, 
then a sequence of 
zero or more N dominoes. 
While $\px^3_2$ is in the interior of this area it performs the 
customary random walk. 
It must interact with $\px^2_1$ and $\px^2_2$  in the 
same way as we have seen other particles interacting above.

So 
\begin{equation}
\begin{aligned}
\px^3_2(t)=\px^3_2(t-1)+\gamma^3_2(t)
&- \mathbf{1}\{\px^3_2(t-1)+\gamma^3_2(t)= \px^2_2(t-1)+1\}\\
&+ \mathbf{1}\{\px^3_2(t-1)+\gamma^3_2(t)= \px^2_1(t-1)\}.
\end{aligned}
\end{equation}

The same pattern repeats itself evermore.
\begin{align}
\px^j_1(t)&=\px^j_1(t-1)+\gamma^j_1(t)
- \mathbf{1}\{\px^j_1(t-1)+\gamma^j_1(t)= \px^{j-1}_1(t-1)+1\}\\
\px^j_j(t)&=\px^j_j(t-1)+\gamma^j_j(t)
+ \mathbf{1}\{\px^j_j(t-1)+\gamma^j_j(t)= \px^{j-1}_{j-1}(t-1)\}\\
\px^j_i(t)&=\px^j_i(t-1)+\gamma^j_i(t)
- \mathbf{1}\{\px^j_i(t-1)+\gamma^j_i(t)= \px^{j-1}_{j}(t-1)+1\}\\
&\quad\quad\quad\quad\quad\quad\quad\quad\ \;
+ \mathbf{1}\{\px^j_i(t-1)+\gamma^j_i(t)= \px^{j-1}_{j-1}(t-1)\}.
\end{align}
with initial conditions
$x^j_i(j)=i$
for $j=2,\ldots$ and $1\leq i\leq j$.

\begin{figure}[htbp]
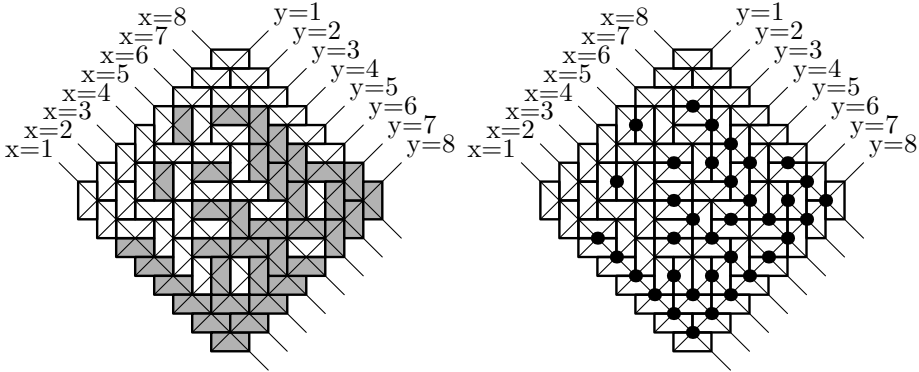

\noindent\includegraphics{pictures/aztec.21.mps}
\includegraphics{pictures/aztec.22.mps}
\caption{Same diamond }
\label{fig:particles}
\end{figure}

In order to analyse this situation it is suitable to perform a 
change of variables, 
\begin{equation}
X_i^j(t)=\px_i^j(t-j), 
\end{equation}
which gives the equations given in section~\ref{sec:particle-process}.

\section{Transition probabilities on two lines}
\label{sec:interl-rand-walks}
In order to analyse the dynamics just described
we follow Warren's example and first consider
just two lines at a time. 
What we do in this section is very similar to section~2 
of~\cite{warren:dyson_brownian_motions}.

Consider the 
$\Wdisc^{n+1,n}$-valued process 
process $(\Qdisc^n(t)) = (X(t),Y(t))$ with 
components $X_1(t), \dots, X_{n+1}(t)$ and $Y_1(t), \dots, Y_{n}(t)$,
satisfying the equations 
\begin{equation}
\label{eq:2}
\begin{aligned}
Y_i(t+1)&=Y_i(t)+\beta_i(t)\\
X_1(t+1)&=X_1(t)+\alpha_1(t) 
- \mathbf{1}\{ X_1(t)+\alpha_1(t)=Y_1(t+1)+1\}\\
X_i(t+1)&=X_i(t)+\alpha_i(t) + 
\mathbf{1}\{
X_i(t)+\alpha_i(t)=Y_{i-1}(t+1)\}\\
&\quad- \mathbf{1}\{X_i(t)+\alpha_i(t)=Y_{i}(t+1)+1\}\\
X_{n+1} (t+1)&=X_i(t)+\alpha_{n+1}(t) + 
\mathbf{1}\{
X_{n+1}(t)+\alpha_{n+1}(t)=Y_{n}(t+1)\}
\end{aligned}
\end{equation}
where  $\alpha_i(t)$ and $\beta_i(t)$ are i.i.d.
coin tosses, 
s.t. $\mathbb{P}[\alpha_i(t)=0]=\mathbb{P}[\alpha_i(t)=1]=\frac{1}{2}$. 
They evolve until the stopping time 
$\tau=\min\left\{
t: Y_i(t)=Y_{i+1}(t)\text{ for some }i\in \{1,\dots,n-1\}\right\}$.
At the time $\tau$ the process is  killed
and remains constant for all time after that. 
This is a very simple dynamics, each $Y_i$ either 
stays or increases one independently of all others. The $X_i$ do
the same but are sometimes pushed up or down by a $Y_{i-1}$
or $Y_{i}$ respectively so as to stay in the cone $\Wdisc^{n,n+1}$.
This is the discrete analog of the process
$\Qcont$ defined in section~\ref{sec:warren} of this paper.

\begin{lemma}
\label{thm:qzero}
For any $f:\Wdisc^{n+1,n}\rightarrow \mathbb{R}$,
\begin{equation}
\sum_{(\xdisc', \ydisc')\in \Wdisc^{n+1,n}} 
\qdisc_0^n((\xdisc, \ydisc), (\xdisc', \ydisc'))f((\xdisc', \ydisc')) 
= f((\xdisc, \ydisc)).
\end{equation}
\end{lemma}
\begin{proof}
Let $m=2n+1$ and $\zdisc_1=\xdisc_1$, $\zdisc_2=\ydisc_1$, 
\dots, $\zdisc_{m-1}=\ydisc_n$,
$\zdisc_{m}=\xdisc_{n+1}$.
Equation (42) in~\cite{warren:dyson_brownian_motions} states that
\begin{equation}
\det
\left\{
\begin{matrix}
\mathbf{1}\{\zdisc_i\leq \zdisc_j'\} & i\geq j\\
-\mathbf{1}\{\zdisc_i\leq \zdisc_j'\} & i<j
\end{matrix}
\right\}
=
\mathbf{1}
\{
\zdisc_1\leq \zdisc_1', 
\zdisc_2\leq \zdisc_2', \dots,
\zdisc_m\leq \zdisc_m' 
\}
\end{equation}
for $\zdisc, \zdisc'\in\Wdisc^n$.
Applying the operator 
$\Delta_{\zdisc_1'}(-\bar\Delta_{\zdisc_2})\Delta_{\zdisc_3'}\dots
(-\bar\Delta_{\zdisc_{m-1}})\Delta_{\zdisc_m'}$
to both sides of that equality turns the left hand side into 
$\qdisc_0^n((\xdisc, \ydisc),(\xdisc', \ydisc'))$ and the right
hand side into 
$\mathbf{1}
\{
\zdisc_1=\zdisc_1'$, 
$\zdisc_2= \zdisc_2'$, \dots,
$\zdisc_m= \zdisc_m' 
\}$.
\end{proof}
\begin{prop}
$\qdisc_t$, for  $t=0,1,\dots$, are the transition probabilities for the 
process $(X,Y)$, i.e. for $(\xdisc,\ydisc)$, $(\xdisc',\ydisc')\in W^{n+1,n}$,
\begin{equation}
\qdisc_t^n((\xdisc,\ydisc), (\xdisc',\ydisc')) = 
\mathbb{P}^{(\xdisc,\ydisc)}[X(t)=\xdisc', Y(t)=\ydisc';
 t <\tau]
\end{equation}
\end{prop}
\begin{proof}

Take some test function $f:\Wdisc^{n+1,n}\rightarrow \mathbb{R}$.
Let 
\begin{equation}
F(t,(\xdisc,\ydisc)):= \sum_{(\xdisc', \ydisc')\in \Wdisc^{n+1,n}}
 \qdisc_t^n((\xdisc, \ydisc), (\xdisc', \ydisc'))f(\xdisc', \ydisc') 
\end{equation}
and 
\begin{equation}
\label{eq:5}
G(t,(\xdisc,\ydisc)):= \mathbb{E}^{(\xdisc,\ydisc)}[f(X_t,Y_t); t<\tau]
\end{equation}
We want of course to prove that $F$ and $G$ are
equal and we will do this by showing that 
they satisfy the same recursion equation 
with the same boundary values.  By 
lemma~\ref{thm:qzero} we already know 
that 
\begin{equation}
\label{eq:16}
F(0, \cdot) \equiv G(0, \cdot) \equiv f(\cdot).
\end{equation}

The master equation satisfied by $G$ is 
\begin{equation}
\label{eq:6}
G(t+1,(\xdisc,\ydisc))=
\frac{1}{2^{2n+1}}\sum_{a_i, b_i\in \{0,1\}}
G(t, \xdisc_1+a_1, \ydisc_1+b_1, \xdisc_2+a_2,\dots, 
\ydisc_{n}+b_n, \xdisc_{n+1}+a_{n+1}).
\end{equation}
This formula simply encodes the dynamics that each particle
either stays or jumps forward one step. 
This needs to be supplemented with some
boundary conditions that have to do with 
  the interactions between particles.

When two of the $\ydisc_i$-particles coincide, 
this corresponds to the event $t=\tau$, 
which does not contribute to the expectation in~(\ref{eq:5}).
Thus 
\begin{equation}\label{eq:7}
G(\tdisc, \dots, \ydisc_{i-1}=\zdisc, \xdisc_i, \ydisc_i=\zdisc,\dots):= 0.
\end{equation}
Also, the particle $\xdisc_i$ cannot jump past $\ydisc_i$,
\begin{equation}\label{eq:8}
G(\tdisc, \dots, \xdisc_i=\zdisc+1, \ydisc_i=\zdisc,\dots)
:= G(\tdisc,\dots, \xdisc_i=\zdisc, \ydisc_i=\zdisc,\dots) 
\end{equation}
and $\xdisc_{i+1}$ must not drop below $\ydisc_i+1$,
\begin{equation}\label{eq:9}
G(\tdisc, \dots, \ydisc_i=\zdisc, \xdisc_{i+1}=\zdisc,\dots)
:= G(\tdisc,\dots, \ydisc_i=\zdisc, \xdisc_{i+1}=\zdisc+1,\dots) .
\end{equation}

$G(t+1,\cdot)$ is uniquely determined from $G(t, \cdot)$ using the 
recursion equation and boundary values above. 
It follows that $G$ is uniquely 
defined by the recursion equation~(\ref{eq:6})
and the boundary conditions (\ref{eq:16},\ref{eq:7},\ref{eq:8},\ref{eq:9}).

Observe that, all functions $g:\mathbb{Z}\rightarrow\mathbb{R}$, 
satisfy
\begin{equation}\label{eq:10}
\frac{1}{2}(g(\xdisc)+g(\xdisc+1))=g(\xdisc)+\frac{1}{2}\bar\Delta g(\xdisc).
\end{equation}
Using this identity many times on~(\ref{eq:6}) shows that
\begin{equation}
G(t+1,(\xdisc,\ydisc))=(1+\frac{1}{2}\bar\Delta_{\xdisc_1})
(1+\frac{1}{2}\bar\Delta_{\ydisc_1})
(1+\frac{1}{2}\bar\Delta_{\xdisc_2})
\cdots (1+\frac{1}{2}\bar\Delta_{y_n})
(1+\frac{1}{2}\bar\Delta_{\xdisc_{n+1}}) 
G(t,\xdisc_1, \ydisc_1, \dots, \ydisc_n, \xdisc_{n+1})
\end{equation}
which can be rewritten as
\begin{equation}
\label{eq:grecursion}
\bar\Delta_t G(t, (\xdisc,\ydisc)) = \left(\prod_{i=1}^{n+1}(1+\frac{1}{2}\bar\Delta_{\xdisc_i})
\prod_{i=1}^n(1+\frac{1}{2}\bar\Delta_{\ydisc_i}) -1\right)G(t, (\xdisc,\ydisc)).
\end{equation}

The boundary conditions can be rewritten in this 
notation as well, equations  (\ref{eq:7},\ref{eq:8},\ref{eq:9})
can be rewritten to
\begin{align}
G(t, (\xdisc,\ydisc))&=0 &&\text{when $\ydisc_i=\ydisc_{i+1}$,}\\
\bar\Delta_{\xdisc_i} G(t, (\xdisc,\ydisc)) &= 0 &&
\text{when $\xdisc_i=\ydisc_{i}$ and}\\
\bar\Delta_{\xdisc_{i+1}} G(t, (\xdisc,\ydisc)) &= 0
 &&\text{when $\xdisc_{i+1}=\ydisc_{i}$.}
\end{align}

Now let us look at $F$. 
The observation~(\ref{eq:10}) gives that 
$\phidisc*\psi= (1+\frac{1}{2}\Delta) \psi$.
In particular, 
$\phidisc^{(n+1)}(\ydisc-\xdisc) = 
 (1+\frac{1}{2}\bar\Delta_\xdisc)\phidisc^{(n)}(\ydisc-\xdisc)$.
\begin{align*}
F(t+1, (\xdisc,\ydisc)) =&
\begin{bmatrix}
\phidisc^{(t+1)}(\xdisc_1'-\xdisc_1)  & 
\Delta^{-1}\phidisc^{(t+1)}(\ydisc_1'-\xdisc_1)-1 &  
\phidisc^{(t+1)}(\xdisc_2'-\xdisc_1) &\dots \\
\Delta \phidisc^{(t+1)}(\xdisc_1'-y_1)  & 
\phidisc^{(t+1)}(\ydisc_1'-\ydisc_1) &  
\Delta\phidisc^{(t+1)}(\xdisc_2'-\ydisc_1) &\dots \\
\phidisc^{(t+1)}(\xdisc_1'-\xdisc_2)  & \Delta^{-1}\phidisc^{(t+1)}(\ydisc_1'-\xdisc_2) &  \phidisc^{(t+1)}(\xdisc_2'-\xdisc_2) &\dots \\
 \vdots
\end{bmatrix}
\\
=&(1+\frac{1}{2}\bar \Delta_{\xdisc_1})\begin{bmatrix}
\phidisc^{(t)}(\xdisc_1'-\xdisc_1)  & \Delta^{-1}\phidisc^{(t)}(\ydisc_1'-\xdisc_1)-1 &  \phidisc^{(t)}(\xdisc_2'-\xdisc_1) &\dots \\
\Delta \phidisc^{(t+1)}(\xdisc_1'-\ydisc_1)  & \phidisc^{(t+1)}(\ydisc_1'-\ydisc_1) &  \Delta\phidisc^{(t+1)}(\xdisc_2'-\ydisc_1) &\dots \\
\phidisc^{(t+1)}(\xdisc_1'-\xdisc_2)  & \Delta^{-1}\phidisc^{(t+1)}(\ydisc_1'-\xdisc_2) &  \phidisc^{(t+1)}(\xdisc_2'-\xdisc_2) &\dots \\
 \vdots
\end{bmatrix}
\\
&\vdots \displaybreak[0]\\
= &\prod_{i=1}^{n+1}(1+\frac{1}{2}\bar\Delta_{\xdisc_i})
\prod_{i=1}^n(1+\frac{1}{2}\bar\Delta_{\ydisc_i}) F(t, (\xdisc,\ydisc))
\end{align*}
which shows that $F$ satisfies the
same recursion~(\ref{eq:grecursion}) as $G$. 
Now let us take a look at the boundary values. 

$\qdisc_t^n((\xdisc,\ydisc), (\xdisc', \ydisc'))$ is zero 
when $\ydisc_i=\ydisc_{i+1}$ because two of its rows are then equal. 
When $\ydisc_i=\xdisc_i$ for some $i$ then 
$\bar\Delta_{\xdisc_i}\qdisc_t^n((\xdisc,\ydisc), (\xdisc', \ydisc')) =0$
because two rows will be equal when you take the difference operator
into the determinant. The same  argument
shows that  $\bar\Delta_{\xdisc_{i+1}}\qdisc_t^n((\xdisc,\ydisc), (\xdisc', \ydisc')) =0$ 
when $\ydisc_i=\xdisc_{i+1}$.
Applying this knowledge to the sum $F$, shows that
\begin{align}
F(t, (\xdisc,\ydisc))&=0 &&\text{when $\ydisc_i=\ydisc_{i+1}$}\\
\bar\Delta_{\xdisc_i} F(t, (\xdisc,\ydisc)) &= 0 &&\text{when $\xdisc_i=\ydisc_{i}$}\\
\bar\Delta_{\xdisc_{i+1}} F(t, (\xdisc,\ydisc)) &= 0 && \text{when $\xdisc_{i+1}=\ydisc_{i}$}
\end{align}

Since $F$ and $G$ satisfy the same recursion equation with the
same boundary values, they must be equal.
\end{proof}

Again, following the example of Warren, we observe
that it is possible to condition the processes never to 
leave $\Wdisc^{n,n+1}$ via a so called Doob $h$\nobreakdash-transform. See 
for example~\cite{konig:non-colliding} for 
details about $h$\nobreakdash-transforms for discrete processes.
Let 
\begin{equation}
h_n(\xdisc)=\prod_{1\leq i< j\leq n} (\xdisc_j-\xdisc_i).
\end{equation}

The $h$\nobreakdash-transform of the process above has transition
probabilities
\begin{equation}
\qdisc_t^{n,+} ((\xdisc,\ydisc), (\xdisc', \ydisc'))= 
\frac{h(\ydisc')}{h(\ydisc)}
\qdisc_t^n((\xdisc,\ydisc), (\xdisc', \ydisc')).
\end{equation}
Just for the sake of notation, call the
transformed process $(\mathpzc{Q}^{n,+}(t))$. 

The idea now is to stitch together the process
$\Xdisc$ from processes $\mathpzc{Q}^{k,+}$
for $k=1$, \dots, $n-1$, just like Warren does in 
the continuous case.  For this we need to 
establish some auxiliary results about $\mathpzc{Q}^{n}$
and $\mathpzc{Q}^{n,+}$. 

One observation to make 
is that it is possible to integrate out the $\xdisc$ variables
from $q_t^{n,+}$. 
\begin{equation}
\pdisc_t^{n,+}(\ydisc, \ydisc'):= 
\int_{\xdisc_1'\leq \ydisc_1'<\dots\leq \ydisc_{n}'<\xdisc_{n+1}'}
\qdisc_t^{n,+}(x,y) \, d\xdisc' = 
\frac{h_n(\ydisc')}{h_n(\ydisc)}
\det[\phidisc^{(t)}(\ydisc_j'-\ydisc_i)]_{1\leq i, j\leq n}
\end{equation}
where $d\xdisc'$ is counting measure on 
$\Wdisc^{n+1}= \{\xdisc\in\mathbb{Z}^{n+1}: \xdisc_1<\dots<\xdisc_{n+1}\}$.  
The reader might recognise $\pdisc_t^{n,+}$, 
as a $h$-transformed version of
the transition probabilities 
from the Lindström-Gessel-Viennot theorem.  
Thus this can be seen as the transition probabilities for 
a process on $\Wdisc^n$, where
all $n$ particles perform independent random 
walks but are
conditioned never to intersect, i.e. never to 
leave $ \Wdisc^n$. We state this as a proposition.

Fix $n>0$ and let $\bar \xdisc=(1,\dots,n+1)\in \mathbb{Z}^{n+1}$ 
and $ \bar \ydisc=(1,\dots,n)\in \mathbb{Z}^{n}$. 

\begin{prop}
\label{thm:qn_to_y}
Consider the process $(\Qdisc^{n,+}(t))=(X(t), Y(t))$ 
 started in $(X(0), Y(0))=(\bar \xdisc, \bar \ydisc)$. The
process $(Y(t))$ is governed by $p^{n,+}$. 
\end{prop}


Now for a technical lemma.
For $\xdisc\in \Wdisc^n$, 
let $\Wdisc^n(\xdisc)=\{y\in \mathbb{R}^n: \xdisc_1\leq y_1 < \dots \leq y_n<\xdisc_{n+1}\}\subset\mathbb{Z}^n$
and for $\ydisc\in \Wdisc^n(\xdisc)$
let 
\begin{equation}
\lambda^n(\xdisc,\ydisc)=n! \frac{h_n(\ydisc)}{h_{n+1}(\xdisc)}.
\end{equation}
It is not a difficult calculation to show that $\lambda^n(\xdisc, \cdot)$ is
a probability measure on $\Wdisc^{n}(\xdisc)$. 
Just rewrite $h_n(y)$ as a Vandermonde matrix and 
perform the summation over all $\ydisc$. 
\begin{lemma}
\label{thm:qn_to_pn_lemma}
\begin{equation}
\int_{W^n(\xdisc)}
\lambda^n(\xdisc,\ydisc)
q_t^{n,+}((\xdisc,\ydisc),(\xdisc',\ydisc')) \; d\ydisc=\pdisc_t^{n+1,+}(\xdisc,\xdisc') \lambda^n(\xdisc', \ydisc')
\end{equation}
where $d\ydisc$ is counting measure.
\end{lemma}
\begin{proof} An elementary calculation given the explicit 
formula for $\qdisc_t^{n,+}$.
\end{proof}

\begin{thm}
\label{thm:qn_to_pn}
Consider the process $(\Qdisc^{n,+}(t))=(X(t), Y(t))$ 
 started in $(X(0), Y(0))=(\bar \xdisc, \bar \ydisc)$. The
process $(X(t))$ is governed by $p^{n+1,+}$. 
\end{thm}
Our  proof of this is very similar to the proof of proposition~5 
in~\cite{warren:dyson_brownian_motions}.
\begin{proof}
Recall the the transition probabilities for $(\Qdisc^{n,+}(t))$
are $\qdisc_\tdisc^{n,+}$. 
\begin{multline*}
\mathbb{P}[X(t_1)\in A_1, \dots, X(t_k)\in A_k ] =\\
\int_{A_1} d\xdisc_1\cdots\int_{A_k}d\xdisc_k 
\int_{\Wdisc^n(\xdisc_1)}d\ydisc_1 \cdots\int_{\Wdisc^n(\xdisc_k)} d\ydisc_k
\qdisc^{n,+}_{t_1}((\bar \xdisc, \bar \ydisc),(\xdisc_1, \ydisc_1))
\qdisc^{n,+}_{t_2-t_1}(( \xdisc_1,  \ydisc_1),(\xdisc_2, \ydisc_2))
\cdots \\
\cdots
\qdisc^{n,+}_{t_k-t_{k-1}}
(( \xdisc_{k-1},  \ydisc_{k-1}),(\xdisc_k, \ydisc_k))=
\end{multline*}
$\Wdisc^n(\bar\xdisc)$ contains but one
element.
\begin{multline*}
\int_{A_1} d\xdisc_1\cdots\int_{A_k}d\xdisc_k 
\int_{\Wdisc^n(\bar\xdisc)}d\ydisc_0
\int_{\Wdisc^n(\xdisc_1)}d\ydisc_1 \cdots\int_{\Wdisc^n(\xdisc_k)} d\ydisc_k\\
\lambda^n(\bar\xdisc, \ydisc_0)
\qdisc^{n,+}_{t_1}((\bar \xdisc, \ydisc_0),(\xdisc_1, \ydisc_1))
\qdisc^{n,+}_{t_2-t_1}(( \xdisc_1,  \ydisc_1),(\xdisc_2, \ydisc_2))
\cdots \\
\cdots
\qdisc^{n,+}_{t_k-t_{k-1}}
(( \xdisc_{k-1},  \ydisc_{k-1}),(\xdisc_k, \ydisc_k))  =
\end{multline*}
Repeated applications of lemma~\ref{thm:qn_to_pn_lemma}
conclude the proof.
\begin{multline*}
\int_{A_1} d\xdisc_1\cdots\int_{A_k}d\xdisc_k 
\int_{\Wdisc^n(\xdisc_1)}d\ydisc_1 \cdots\int_{\Wdisc^n(\xdisc_k)} d\ydisc_k
\pdisc^{n,+}_{t_1}(\bar \xdisc,\xdisc_1)
\lambda^n(\xdisc_1, \ydisc_1)
\qdisc^{n,+}_{t_2-t_1}(( \xdisc_1,  \ydisc_1),(\xdisc_2, \ydisc_2))
\cdots \\
\cdots
\qdisc^{n,+}_{t_k-t_{k-1}}
(( \xdisc_{k-1},  \ydisc_{k-1}),(\xdisc_k, \ydisc_k))  =
\end{multline*}
\begin{multline*}
\int_{A_1} d\xdisc_1\cdots\int_{A_k}d\xdisc_k 
\int_{\Wdisc^n(\xdisc_k)} d\ydisc_k
\pdisc^{n,+}_{t_1}(\bar \xdisc,\xdisc_1)
\pdisc^{n,+}_{t_2-t_1}(\xdisc_1,\xdisc_2)
\cdots 
\pdisc^{n,+}_{t_k-t_{k-1}}
( \xdisc_{k-1},\xdisc_k)  
\lambda^n(\xdisc_k, \ydisc_k)=\\
\int_{A_1} d\xdisc_1\cdots\int_{A_k}d\xdisc_k 
\pdisc^{n,+}_{t_1}(\bar \xdisc,\xdisc_1)
\pdisc^{n,+}_{t_2-t_1}(\xdisc_1,\xdisc_2)
\cdots 
\pdisc^{n,+}_{t_k-t_{k-1}}
( \xdisc_{k-1},\xdisc_k)  
\end{multline*}
\end{proof}

\section{Transition probabilities for the Aztec Diamond Process}
\label{sec:interl-from-aztec}
Let us now return to the process $(\Xdisc(t))$ that came from 
the shuffling algorithm.
Observe in the recursions~(\ref{eq:1}), 
the formulas that define $X^{k+1}$ 
contain $X^k$ but not $X^j$ for $j<k$. 
Thus $X^k$ is conditionally independent 
of $(X^1,\dots,X^{k-1})$ given $X^k$. 
Also, the dependence of $X^{k+1}$ on $X^k$ 
is the same as the dependence of 
$X$ on $Y$ in $\Qdisc^{k}$ and in $\Qdisc^{k,+}$, see~(\ref{eq:2}). 
This, together with theorem~\ref{thm:qn_to_pn} lends itself 
to an inductive procedure for constructing the process
$\Xdisc$.
\begin{enumerate}
\item
The process $(X^k(t), t=0,1,\dots)$ is started
at $X^k(0)=\bar x^k$ and has transition probabilities
governed by $\pdisc^{k,+}$ for $k=1$, $2$, \dots, $k$,
\item
By proposition~\ref{thm:qn_to_y}, $X^k$
can be considered as the $Y$ component of the process $Q^{k,+}$.
By the observation above, 
the pair of processes $(X^k(t),X^{k+1}(t))$
has the same distribution as $Q^{k,+}$
started at $(\bar x^k, \bar x^{k+1})$ and
are thus governed by transition probabilities $\qdisc^{k,+}$,
\item
The process $X^{k+1}$ is conditionally independent of $(X^1,\dots, X^{k-1})$ 
given $X^{k}$. 
\item
By  theorem~\ref{thm:qn_to_pn}, the process $X^{k+1}$ 
is governed  by transition probabilities $\pdisc^{k+1,+}$
and started at $X^{k+1}(0)=\bar x^{k+1}$.
\end{enumerate}

This proves theorem~\ref{thm:transition}.

\section{Asymptotics}
\renewcommand{\tdisc}{{\tilde t}}
\renewcommand{\xdisc}{{\tilde x}}
\renewcommand{\ydisc}{{\tilde y}}

\label{sec:asymptotics}
We shall now see  some results that
lend support to our conjecture that Warren's 
process $\Xcont$ can be recovered
as a scaling limit of $\Xdisc$, the process from the Aztec diamond.
Let us rescale time by $\tdisc=N\tcont$ 
and space by 
$\xdisc_i = \frac{1}{2}N\tcont+ \frac{1}{2}\sqrt{N}\xcont_i$
for $i=0,\dots,n$
in the above processes. 

First let us show how to recover the entrance law for Warrens process.
Recall that the Aztec diamond process the discrete process starts 
at $X^n(0)=\bar x^n$ and
 $X^{n+1}(0)=\bar x^{n+1}$.

\begin{lemma}
\label{thm:entrancelaw}
\begin{equation}
\left(\frac{\sqrt{N}}{2} \right)^{2n+1}
\qdisc_\tdisc^{n,+}( (\bar x^n, \bar x^{n+1}), (\xdisc, \ydisc))
\rightarrow 
\nu^n_t(x, y) 
\end{equation}
as $N\rightarrow\infty$ where $\nu^n_t(x, y) $ is given
by~(\ref{eq:3}).
\end{lemma}
\begin{proof}
Observe that we can write 
\begin{equation}
2^{-2n-1}\sqrt{N}^{2n+1}
\qdisc_\tdisc^{n,+}( (\bar x^{n+1}, \bar x^n), (\xdisc, \ydisc))
=2^{-2n+1}\sqrt{N}^{2n+1}
\int_{\bar x^n\in\Wdisc(\bar x^{n+1})}
\lambda^n(\bar x,\bar y)
\qdisc_\tdisc^{n,+}( (\bar x, \bar y), (\xdisc, \ydisc)) \; d\bar x^n
\end{equation}
where $d\bar x^n$ is 
counting measure on the space $\Wdisc(\bar x^{n+1})$ which 
has only one element.
Then we apply lemma~\ref{thm:qn_to_pn_lemma}.
\begin{equation}
=2^{-2n+1}\sqrt{N}^{2n+1}
\pdisc_\tdisc^{n+1,+}( \bar x^{n+1},\xdisc)
\lambda^n(\xdisc,\ydisc)
\end{equation}
which can be written explicitly
as 
\begin{equation}
=2^{-2n+1}\sqrt{N}^{2n+1}n!
\frac{h_n(\ydisc)}{h_{n+1}(\bar x^{n+1})}
\det\left[2^{-\tdisc} \binom{\tdisc}{\xdisc_i-j}\right]_{1\leq i,j\leq n+1}
\end{equation}
which can be evaluated by a formula of Krattenthaler (Theorem 26 of~\cite{krattenthaler:advanced}). 
Applying Stirling's approximation to the result shows our theorem
with $Z_{n}=(2\pi)^{n/2}\prod_{j<n}j!$.
\end{proof}

Likewise, Warren's expression 
for $q^n_t$ can be recovered as
a scaling limit from our expression for $\qdisc_t^n$.
By Stirling's approximation, 
\begin{align}
\frac{1}{2}\sqrt{N}
\phidisc^{(\tdisc)} (  \xdisc ) 
&\rightarrow \phicont_t( \xcont), &
\Delta^{-1}
\phidisc^{(\tdisc)} (  \xdisc ) 
&\rightarrow \int_{-\infty}^x\phicont_t( \ycont)\, d\ycont
\end{align}
and 
\begin{equation}
\frac{1}{4}N\Delta
\phidisc^{(\tdisc)} (  \xdisc ) 
\rightarrow \frac{d}{d\xcont}\phicont_t( \xcont)
\end{equation}
uniformly on compact sets as $N\rightarrow\infty$ 
where 
\begin{equation}
\phicont_t(x)=\frac{1}{\sqrt{2\pi t}}e^{-\frac{x^2}{2t}}.
\end{equation}

\begin{lemma}
\label{thm:transition_density}
For $\xdisc_i=\frac12Ns +\frac{1}{2} \sqrt{N}\xcont_i$,
$\xdisc_i'=\frac12N(s+t) +\frac{1}{2} \sqrt{N}\xcont_i'$,
and the same relations for $\ydisc$ and $\ydisc'$,
\begin{equation}
(\frac{\sqrt{N}}{2})^{2n+1}
\qdisc_{Nt}^n ((\xdisc, \ydisc) , (\xdisc', \ydisc')) \rightarrow
\qcont_t^n((\xcont, \ycont), (\xcont', \ycont'))
\end{equation}
uniformly on compact sets as $N\rightarrow\infty$.
\end{lemma}
\begin{proof}
Just insert the limit relations given above for $\phidisc$ and
$\phicont$
in the explicit expression for $\qdisc^n_t$.
\end{proof}

Finally, one of the  main results of this article is the
following.

\begin{thm}
\label{thm:Xdisc_to_Xcont}
The process $(X^k(\tdisc), X^{k+1}(\tdisc))$ from $\Xdisc$, extended
by interpolation to non-integer times $\tdisc$,
converges in the sense of finite dimensional distributions 
to the process  $(X^k(t), X^{k+1}(t))$ from $\Xcont$.
\end{thm}
\begin{proof}
For times $t_1$, \dots, $t_m$, and 
compact sets $A_1$, \dots, $A_m\in \Wcont^{n,n+1}$, 
we need to study
\begin{multline}
\lim_{N\rightarrow\infty}
\int_{A_1} d\xdisc_1 d\ydisc_1
\cdots
\int_{A_m} d\xdisc_m d\ydisc_m\times \\
\times\qdisc^n_{\tdisc_1} ((\bar x, \bar y), ( \xdisc_1, \ydisc_1))
\qdisc^n_{\tdisc_2-\tdisc_1} (( \xdisc_1,  \ydisc_1), ( \xdisc_2, \ydisc_2))
\cdots
\qdisc^n_{\tdisc_m-\tdisc_{m-1}}((\xdisc_{m-1},  \ydisc_{m-1}), (\xdisc_m,\ydisc_m))
\end{multline}
where of course $d\xdisc_id\ydisc_i$ 
is point measure on $\Wdisc^{n, n+1}$. 
This is a case of a Riemann-sum converging to an integral. 
By the uniform convergence of  $\qdisc^n_\tdisc$ 
in lemmas~\ref{thm:entrancelaw} and~\ref{thm:transition_density}, 
we can interchange the order of integration
and taking the limit. This gives 
\begin{multline}
=\int_{A_1} d\xcont_1 d\ycont_1
\cdots
\int_{A_k} d\xcont_k d\ycont_k\times\\
\times\nu^n_{t_1} ( \xcont_1, \ycont_1)
\qcont^n_{t_2-t_1} (( \xcont_1, \ycont_1), ( \xcont_2, \ycont_2))
\cdots
\qcont^n_{t_k-t_{k-1}}((\xcont_{k-1},\ycont_{k-1}), (\xcont_k,\ycont_k))
\end{multline} 
which proves our theorem.
\end{proof}

Theorem \ref{thm:dyson_bm} follows from 
theorem~\ref{thm:Xdisc_to_Xcont} by just restricting to $(X^k)$. 
We feel that given this theorem together 
with the fact that $X^{k+1}$ is conditionally independent of
$X^{1},\dots,X^{k-1}$ given $X^k$, lends a lot of credibility 
to the conjecture in the introduction.

We can also say something about the limit at a fixed time.
Let $\Kdisc$ be the cone of points $x= (x^1, \dots, x^n) $
with $x^k=(x^k_1,\dots,x^k_k)\in \mathbb{Z}^k$ 
such that
\begin{equation}
x^{k+1}_{i} \leq x^{k}_i < x^{k+1}_{i+1}.
\end{equation} 
For each $x^n\in \Wdisc^n$ we will denote by $\Kdisc(x^n)$ 
the set of all $(x^1,\dots, x^{n-1})$ such that $(x^1,\dots, x^{n-1}, x^n)\in \Kdisc$.
The number of points in $\Kdisc(x^n)$ is 
\begin{equation}
\card(\Kdisc(x^n)) = \frac{h_n(x^n)}{\prod_{k<n} k!}.
\end{equation}

It follows from the characterisation in section~\ref{sec:interl-from-aztec}
that at a fixed time the distribution of $X^{n-1}(t)$ given 
$X^n(t)$ is $\lambda^{n-1}(X^n(t) , \cdot ) $. 
Together with the conditional independence noted in that section
this implies that the distribution of $(X^1(t), \dots, X^{n-1}(t))$
given $X^n(t)$ is uniform in $\Kdisc(X^n(t))$. 
So the probability distribution of $\Xdisc(t)$ 
is 
\begin{equation}
\label{eq:13}
m^n_t ( x )  = 
\pdisc_t^n (\bar x^n, x^n) 
\frac{\chi(x^1, x^2) \dots\chi(x^{n-1}, x^n)}{\card(\Kdisc(x^n ))  }
\end{equation}
where $\chi(x^k,x^{k+1})$ 
is one iff $ x^{k+1}_{i} \leq x^{k}_i < x^{k+1}_{i+1}$
for all $i=1$, \dots, $k-1$ and zero otherwise.

For $x^n\in \Wcont^n$,  define $\Kcont(x^n)$
as the set of $(x^1, \dots, x^{n-1})$ where $x^k\in \mathbb{R}^k$
 satisfying   $ x^{k+1}_{i} \leq x^{k}_i \leq x^{k+1}_{i+1}$.
The $n(n-1)/2$-dimensional volume of $\Kcont(x^n)$ 
is 
\begin{equation}
\vol(\Kcont(x^n)) = \frac{h_n(x^n)}{\prod_{k<n} k!}.
\end{equation}

\begin{thm}
Consider the process $(\Xdisc(t))_{t\in \mathbb{Z}^+}$ under the rescaling
\begin{equation}
\tilde X^{k}_i = \frac{X^{k}_i(N)  - \frac{1}{2} N} {\frac12 \sqrt{N}}.
\end{equation}
As $N\rightarrow \infty$, $\tilde X\rightarrow \Lambda$ weakly where
$\Lambda$ has distribution
\begin{equation}
\mu^n_1(x^n) 
\frac{\chi(x^1, x^2)  \dots \chi(x^{n-1}, x^n)  }{\vol(\Kcont(x^n)}.
\end{equation}
\end{thm}
The expression for $\mu_1^n$ is given in equation~(\ref{eq:14}).
\begin{proof}
Put in the correct rescaling in~(\ref{eq:13}) and perform a computation
that is practically the same as that in the proof of~\ref{thm:entrancelaw}.
\end{proof} 
This distribution $\Lambda$ also happens to be the distribution 
of $\Xcont(1)$, which is consistent with our conjecture. 
It has been studied~\cite{baryshnikov:gue_queu} and is the distribution
of the GUE minor process mentioned in the introduction. 
This proves theorem~\ref{thm:GUE_process}.

\renewcommand{\tdisc}{t}
\renewcommand{\xdisc}{x}
\renewcommand{\ydisc}{y}

\section{Closing Remarks}

Looking at the expression of transition probabilities $\qdisc^{n,+}_t$
it is natural to ask the question, what happens if we plug in a
different $\phidisc$  than $\frac12(\delta_0+\delta_1)$
into that determinantal formula?
It turns out that for many other $\phi$ this gives 
a valid transition probability, although 
we do not fully understand why the Doob $h$-conditioning
still works in that case.  
It would be interesting to see sufficient and necessary conditions
on $\phi$ for this construction to work.

We should mention an article by Dieker and
Warren, \cite{dieker:determinantal}.
They study only  the top and bottom particles separately 
from our model, i.e. in our language
$(X^1_1(t), X^2_2(t), \dots, X^{n}_n(t))$ 
and $(X^1_1(t), X^2_1(t), \dots, X^{n}_1(t))$ from $\Xdisc(t)$.
They consider both geometric jumps 
($\phidisc=(1-q)(\delta_0+q\delta_1+q^2\delta^2+\dots)$, for $0<q<1$) and 
Bernoulli jumps ($\phidisc=p\delta_0 + q\delta_1$ where $p+q=1$). 
They write down transition probabilities but do not
do the rescaling to obtain a process in continuous time and space.

Another reference  worthy of attention is 
\cite{johansson:multidimenstional}, by Johansson.
He considers only
geometric jumps and studies only the top particles 
$(X^1_1(t), X^2_2(t), \dots, X^{n}_n(t))$ 
from our model,  with a slight change of 
variables that is of no real importance.
He not only writes down transition probabilities, 
but also  recovers the top particles Warren's process $\Xcont$ 
as  the limit of his process properly rescaled.

All the results proved in this article can be generalised to 
$\phidisc=p\delta_0 + q\delta_1$ where $p+q=1$. 
It is also not very difficult given my calculations
to write down transition probabilities for the top particles
and to rescale that to obtain the top particles in Warren's continuous process,
analogous to Johansson \cite{johansson:multidimenstional} but 
with Bernoulli jumps. 
We have not written included that calculation here 
 since we don't think it ads much to our knowledge of these processes,
but it is a fact that adds to the plausability of  the conjecture 
of this article.

\clearpage
\bibliography{art}

\end{document}